\documentclass[12pt]{article}

\usepackage{amsfonts}
\usepackage{amscd}
\usepackage{amsmath}
\usepackage{epsfig}
\usepackage{amsthm}
\usepackage{amssymb}
\usepackage{amsbsy}
\usepackage{latexsym}
\usepackage{eucal}
\usepackage{mathrsfs,bm}
\usepackage{mathtools}
\usepackage{tikz}
\usepackage[colorlinks]{hyperref}
\usepackage{color}

 \newtheorem{theorem}{Theorem}[section]

\newtheorem{corollary}[theorem]{Corollary}
\newtheorem{remark}[theorem]{Remark}
\newtheorem{lemma}[theorem]{Lemma}

\def\be{\begin{equation}}
\def\ee{\end{equation}}
\def\ben{\begin{displaymath}}
\def\een{\end{displaymath}}
\def\baa{\begin{eqnarray}}
\def\eaa{\end{eqnarray}}

\def\ba{\begin{array}}
\def\ea{\end{array}}
\renewcommand{\leq}{\leqslant}
\renewcommand{\geq}{\geqslant}

\newcommand{\Tr}{\operatorname{Tr}}

\makeatletter
\@addtoreset{equation}{section}
\makeatother

\begin{document}
\title{Determinant of Friederichs Dirichlet Laplacians on $2$-dimensional hyperbolic cones}

\author{Victor Kalvin}

\date{}
\maketitle

\begin{abstract} We explicitly express the spectral determinant of Friederichs Dirichlet Laplacians on the $2$-dimensional hyperbolic (Gaussian curvature $-1$) cones in terms of the cone angle and the geodesic radius of the boundary.  

\end{abstract}


\section{Introduction and main results}\label{S1}
In this paper we find an explicit formula for the spectral determinant of Friederichs Dirichlet Laplacians on the $2$-dimensional hyperbolic (constant Gaussian curvature $-1$) cones in terms of the cone angle and the geodesic radius of the boundary.  The result is of interest in theoretical physics (in connection with study of quantum Hall effect on singular surfaces,  e.g.~\cite{Klevtsov}), geometric analysis, mathematical physics, number theory, and arithmetic geometry  (see e.g.~\cite{F-P,TZ,Teo}, where hyperbolic cones appear due to elliptic generators of a Fuchsian group).  In particular, in~\cite{F-P} the statement and proof of Riemann-Roch isometry in the non-compact orbifold setting is based, among many other things, on direct calculation of  asymptotics for the determinant on a shrinking model hyperbolic cone (see Theorem~2.5 and its lengthy  and highly technical proof in Section 7 op.cit.). 
The asymptotics~\cite[Theorem 2.5]{F-P} is unfortunately incorrect even in the simplest case with no conical singularity (see Remark~\ref{pol-alv} at the end of this paper). We obtain the correct one as an immediate consequence of our  formula for the determinant.

In this paper we employ an ad-hoc and rather simple method:  we first obtain auxiliary results for spherical  cones and then pass from spherical (positive constant curvature) to  hyperbolic (negative constant curvature) cones by means of analytic continuation with respect to the curvature. In particular, when passing through the curvature zero case we independently obtain the well-known results for the flat cone~\cite{Spreafico}. The  results for spherical cones are retrieved from those  for the spectral determinant on a spindle~\cite{Spreafico3} with the help of integral representation for the Barnes double zeta function~\cite{Spreafico2},  the BFK  decomposition formula~\cite{BFK}, its modification for surfaces with boundary~\cite{LeeBFKboundary}, and the classical Polyakov-Alvarez formula~\cite{Alvarez,OPS}. 

In the geodesic polar coordinates $(r,\theta)$ centered at the vertex,  the metric $m_a$ of a hyperbolic cone of angle $2\pi a$  is given by 
\begin{equation}\label{ma}
m_a=dr^2+a^2\sinh^2 r\,d\theta^2,\quad a>0,
\end{equation}
where  $r\in [0,\eta]$, $\eta$ is the geodesic radius of the boundary, and $\theta\in \Bbb R/2\pi\Bbb Z$. On the hyperbolic cone $\left([0,\eta]_r\times[0,2\pi)_\theta, m_a \right)$ the corresponding Laplace-Beltrami operator $\Delta_a$, initially defined on the smooth functions supported outside of the vertex $r=0$ and satisfying the Dirichlet boundary condition on the circle $r=\eta$,  is not essentially selfadjoint. We pick the Friederichs selfadjoint extension, which we denote by $\Delta^D_{a}\!\!\restriction_{r\leq \eta}$. 
 The spectrum of $\Delta^D_{a}\!\!\restriction_{r\leq \eta}$ consists of positive eigenvalues of finite multiplicity and the spectral $\zeta$-regularized determinant $\det(\Delta^D_{a}\!\!\restriction_{r\leq \eta})$ of  $\Delta^D_{a}\!\!\restriction_{r\leq \eta}$ can be introduced in the usual way.

The main result of this paper is the following theorem. 
\begin{theorem} \label{main}For the spectral determinant of the Friederichs selfadjoint extension of the Laplace-Beltrami operator on the hyperbolic cone $\left([0,\eta]_r\times[0,2\pi)_\theta, m_a \right)$ with Dirichlet boundary condition on $r=\eta>0$ and the metric $m_a$ given in~\eqref{ma} we have 
\begin{equation}\label{star}
\begin{aligned}
\log\det(\Delta^D_{a}\!\!\restriction_{r\leq \eta})=-\frac 1 {6}\left(a+\frac 1 a \right)\log\tanh\frac \eta 2+\frac {3-8\cosh \eta}{12}a-2\zeta_B'(0;a,1,1)
\\
-\frac 1 6\left(a+3+\frac 1 a \right)\log a -\frac 1 2 \log 2\pi.
\end{aligned}
\end{equation}
Here $\zeta_B'$ stands for the derivative with respect to $s$ of  the Barnes double zeta function 
\begin{equation}\label{BZ}
\zeta_B(s;a,b,x)=\sum_{m,n=0}^\infty (am+bn+x)^{-s},\quad \Re s>2, a>0, b>0, x>0,
\end{equation}
that is first defined by the double series and then extended by analyticity to $s=0$.  
\end{theorem}

Since the orbifold setting  (where $a$ can only be the reciprocal of a positive integer) is of particular interest~\cite{F-P,TZ,Teo}, we also formulate the following corollary of Theorem~\ref{main}:
\begin{corollary}\label{cor}
Let $a$ in the statement of Theorem~\ref{main} be such that $a=w^{-1}$, where $w$ is a positive integer. Then the equality~\eqref{star} takes the form 
\begin{equation}\label{w}
\begin{aligned}
\log\det(\Delta^D_{w^{-1}}\!\!\restriction_{r\leq \eta})=&-\frac 1 {6}\left(w+\frac 1 w \right)\log\tanh\frac \eta 2 +\frac {3-8\cosh \eta}{12w}-\frac 2 w \zeta'_R(-1)
\\
&+\frac 2 w\sum_{j=1}^{w-1} j \log\Gamma\left(\frac j w\right)-\frac{w}2 \log 2\pi
+\frac 1 6\left(w+3+\frac 2 w \right)\log w,
\end{aligned}
\end{equation}
where $\zeta'_R(s)$ stands for the derivative with respect to $s$ of the Riemann zeta function $\zeta_R(s)$ and $\Gamma$ is the usual gamma function. 
 \end{corollary}

In particular, by letting $\eta$ go to zero in~\eqref{w}    we obtain the asymptotics
\begin{equation}\label{correction}
\begin{aligned}
\log\det(\Delta^D_{w^{-1}}\!\!\restriction_{r\leq \eta})=&-\left(\frac w 6 +\frac 1 {6w} \right)\log{ \eta}- w\left(\frac 1 3\log 2  +\frac 1 2 \log\pi\right)\\&-\frac 1 w \left( 2  \zeta'_R(-1)- 2\sum_{j=1}^{w-1} j \log\Gamma\left(\frac j w\right)+\frac 5 {12}-\frac 1 6\log 2\right)  \\
&
+\frac 1 2\log w + \frac 1 6 w\log w +\frac 1 3 \frac{\log w } w  +o(1),\quad \eta\to0^+,
\end{aligned}
\end{equation}
which is our correction to the asymptotics obtained in~\cite[Theorem 2.5 and its lengthy and highly technical proof in Section 7]{F-P}, where only the first term and the terms $\frac 1 2\log w$ and  $ \frac 1 6 w\log w $ are in agreement with our results; see also Remark~\ref{pol-alv} at the end of this paper. Let us also mention that the  main idea of the paper~\cite{F-P} on replacing the cusps and hyperbolic cones with smooth caps, relying on the BFK decomposition formula and conformal invariance of its certain components, does not appear to be completely new. For instance,  in~\cite{Kokot} the conical points of polyhedral surfaces were smoothed in a similar way, in~\cite{HKK1} a similar idea was used to close the cylindrical ends, and in~\cite{HKK2} --- to close the Euclidean and conical ends.

This paper consists of three sections. Introduction and the statement of the main results (Theorem~\ref{main} and Corollary~\ref{cor}) occupy this Section~\ref{S1}. In Section~\ref{SPH} we deduce a formula for the determinant of Friederichs Dirichlet Laplacian on spherical cones. 
In Section~\ref{HYP} we further develop our results for the determinant on spherical cones by gluing spherical cones to spherical annuli, then extend our results to hyperbolic cones by means of analytic continuation,  and finally prove Theorem~\ref{main} and Corollary~\ref{cor}.  In Remark~\ref{pol-alv} at the end of Section~\ref{HYP} we discuss the simplest special case of hyperbolic cone without conical singularity (cone of angle $2\pi$, i.e. $a=w^{-1}=1$).  

\section{Preliminaries: spherical cones}\label{SPH}
Consider a spindle of curvature $K>0$ with two conical singularities of angle $2\pi a>0$. This is a surface isometric to the Riemann sphere $\Bbb CP^1$ endowed with the  metric 
\begin{equation}\label{MAK}
m_{a,K}=\frac {4a^2|z|^{2a-2}|dz|^2}{(1+K|z|^{2a})^2}
\end{equation}
of constant curvature $K>0$ and  two antipodal conical singularities of angle $2\pi a$ at $z=0$ and $z=\infty$, see~\cite{Troyanov}. 

The metric~\eqref{MAK}
  is invariant under the inversion $z\mapsto  (\sqrt[a]{K}z)^{-1}$ with respect to the equator $| z|= 1/\sqrt[2a]{K}$ of the spindle. 
In other words, the spindle  consists of two congruent spherical (constant curvature $K>0$) cones glued along the  equator. The first cone is isometric to the disk $|z|\leq  1/\sqrt[2a]{K}$ endowed with the metric~\eqref{MAK}, and the second one is isometric to the ``disk'' $\{z: |z|\geq  1/\sqrt[2a]{K} \}\cup\{\infty\}$ endowed with the same metric. The isometry of those two spherical cones is given by the inversion map.  

We first find the spectral zeta function for the Laplacian on a spindle (Lemma~\ref{LI} below) and then cut the spindle into two spherical cones (Lemma~\ref{LII} below) by using the BFK  decomposition formula~\cite[Theorem B$^*$]{BFK}. 
\begin{lemma} \label{LI} Let $\Delta_{a,K}$ stand for the Friederichs extension of the Laplace-Beltrami operator on the spindle $(\Bbb CP^1, m_{a,K})$.  For the modified (i.e. with zero eigenvalue excluded) spectral zeta function of  the Laplacian $\Delta_{a,K}$ we have
 $$
\begin{aligned}
\zeta'(0,\Delta_{a,K})= 4\zeta_B'(0;a,1,1)-\frac a 2 +\frac 1 3 \left(a+\frac 1 a\right)\log \frac a{\sqrt K} +\log K.
\end{aligned}
$$
\end{lemma}
\begin{proof} This lemma is essentially a reformulation of results in~\cite{Spreafico3}.
By~\cite[Theorem 4.16]{Spreafico3}  for the modified spectral zeta function of $\Delta_{a,1}$ we have
\begin{equation}\label{SPR1}
\begin{aligned}
\zeta'(0, \Delta_{a,1})=-\left(\frac a 3 +\frac 1 {3a}\right)\log a-2\log2\pi+\frac a 3 +1 +\frac 1 {2a} +\log\Gamma\left(1+\frac 1 a \right)
\\
-2a\zeta_R'(-1) -2a\zeta'_H\left(-1,1+\frac 1 a\right)+2i\int_0^\infty\log\frac {\Gamma\bigl(1+i\frac {y}a\bigr)\Gamma\bigl(1+\frac{1+iy} a\bigr)}{\Gamma\bigl(1-i\frac y a\bigr)\Gamma\bigl(1+\frac{1-iy} a\bigr)}\frac {dy}{e^{2\pi y}-1}.
\end{aligned}
\end{equation}
A typo in the  fifth term of the corresponding formula in~\cite[Theorem 4.16]{Spreafico3} was corrected in~\cite{Klevtsov}, where it was also observed that the integral representation for the Barnes double zeta function
$$\begin{aligned}
\zeta'_B(0;a,b,x)&=\left(-\frac 1 2 \zeta_H\left(0,\frac x a\right)+\frac a b \zeta_H\left(-1,\frac x a \right)-\frac 1 {12}\frac b a\right)\log a +\frac 1 2 \log\Gamma\left(\frac x a\right)
\\
&-\frac 1 4\log(2\pi)-\frac a b \zeta_H\left(-1,\frac x a\right)-\frac a b \zeta'_H\left(-1,\frac x a \right)+i\int_0^\infty\log\frac {\Gamma\left(\frac {x+iby}{a}\right)}{\Gamma\left(\frac {x-iby}{a}\right)}\frac {dy}{e^{2\pi y}-1}
\end{aligned}
$$
from~\cite[Proposition 5.1]{Spreafico2} allows to rewrite~\eqref{SPR1}  in the form
\begin{equation}\label{19:48}
\zeta'(0, \Delta_{a,1})=4\zeta_B'(0;a,1,1)-\frac a 2 +\frac 1 3 \left(a+\frac 1 a\right)\log a. 
\end{equation}
By~\cite[Theorem 4.15]{Spreafico3}  we  have
\begin{equation}\label{19:27}
\zeta(0,\Delta_{a,1})=\frac 1 6 \left( a +\frac 1 {a}\right)-1.
\end{equation}
Since the spindles $(\Bbb CP^1, K^{-1}\cdot m_{a,1})$ and $(\Bbb CP^1, m_{a,K})$ are isometric (the isometry is given by the change of variable $z\mapsto \sqrt[2a]{K} z$), 
 the standard rescaling property implies
 $$
\begin{aligned}
\zeta'(0,\Delta_{a,K})=& \zeta'(0, \Delta_{a,1})- \zeta(0,\Delta_{a,1}) \log K.
\end{aligned}
$$
This together with~\eqref{19:48} and~\eqref{19:27} proves the assertion. 
\end{proof}

\begin{lemma} \label{LII}Consider the spherical cone  $( |z|\leq  1/\sqrt[2a]{K},m_{a,K})$ with conical singularity of angle $2\pi a>0$ at $z=0$ and of constant Gaussian curvature  $K>0$; here the metric  $m_{a,K}$ is the same as in~\eqref{MAK}.  
For the spectral zeta function of the corresponding Friederichs Laplacian $\Delta^D_{a,K}\!\!\restriction_{|z|\leq \frac 1 {\sqrt[2a] K}}$ with Dirichlet boundary condition on $| z|=  1/\sqrt[2a]{K}$ we have:
$$
\begin{aligned}
\zeta'(0,\Delta^D_{a,K}\!\!\restriction_{|z|\leq \frac 1 {\sqrt[2a] K}})=& 2\zeta_B'(0;a,1,1)-\frac a 4  
\\
& +\frac 1 6 \left(a+3+\frac 1 a\right)\log a - \frac 1 {12}\left(a+\frac 1 a \right)\log K  +\frac 1 2  \log {2\pi }.
\end{aligned}
$$
\end{lemma}

\begin{proof} The assertion is an immediate consequence of Lemma~\ref{LI} and BFK decomposition formula~\cite[Theorem B$^*$]{BFK}. (Let us mention that for the constant curvature metrics with conical singularities the proof of BFK formula requires only  minor modifications,  details can be found in~\cite[Section 2.3]{Kalvin Pol-Alv}).

Indeed, for the decomposition of  the spindle along its equator  the BFK formula gives:
\begin{equation}\label{BFKspindle}
\det\Delta_{a,K}= \frac {4\pi a}{K} \cdot\left(\det (\Delta^D_{a,K}\!\!\restriction_{|z|\leq \frac 1 {\sqrt[2a] K}})\right)^2\cdot\frac {\det ( N_{a,K}\!\!\restriction_{|z|=\frac 1 {\sqrt[2a] K}})
}{\ell_{a,K}},
\end{equation}
where $\frac {4\pi a}{K}$ is the total area of  the spindle $(\Bbb CP^1, m_{a,K})$, $\ell_{a,K}$ is the length of the equator  in the metric $m_{a,K}$, and $N_{a,K}\!\!\restriction_{|z|=\frac 1 {\sqrt[2a] K}}$ is the Neumann jump operator on the equator (a first order classical pseudodifferential operator). On the surfaces without boundary the quotient $\det ( N_{a,K}\!\!\restriction_{|z|=\frac 1 {\sqrt[2a] K}})
/{\ell_{a,K}}$ is conformally invariant, see e.g.~\cite{Wentworth}. Hence 
$$
\frac {\det ( N_{a,K}\!\!\restriction_{|z|=\frac 1 {\sqrt[2a] K}})}{\ell_{a,K}}=\frac {\det( N_{1,1}\!\!\restriction_{|z|=1})}{\ell_{1,1}}=\frac 1 2,
$$ 
where the last equality is a consequence of~\eqref{BFKspindle} with $a=K=1$ and the explicit formulas~\cite{Weisberger} for the  determinant $\det\Delta_{1,1}$ of the Laplacian on a unit sphere and the determinant $\det (\Delta^D_{1,1}\!\!\restriction_{|z|\leq 1})$ of the Dirichlet Laplacian on a unit hemisphere. 

Being rewritten in terms of the spectral zeta functions the formula~\eqref{BFKspindle} gives
$$
-\zeta'(0,\Delta_{a,K})= \log \frac {4\pi a}{K}- 2\zeta'(0, \Delta^D_{a,K}\!\!\restriction_{|z|\leq \frac 1 {\sqrt[2a] K}}) -\log 2.
$$
This together with Lemma~\ref{LI} implies the assertion.
\end{proof}

\section{From spherical to hyperbolic cones}\label{HYP}
 In this section we prove Theorem~\ref{main} and Corollary~\ref{cor}. The proofs are preceded by Lemma~\ref{LIII} below, where we first further develop our results for spherical cones  and then 
 extend the results to hyperbolic cones by means of analytic continuation with respect to the curvature.  
 \begin{lemma} \label{LIII}
 Consider the   cone  $( |z|\leq 1, m_{a,K})$ of constant curvature $K>-1$ with conical singularity of angle $2\pi a$ at the vertex $z=0$, where $a>0$ and the metric  $ m_{a,K}$ is given in~\eqref{MAK}.  
For the spectral zeta function of the corresponding Friederichs Laplacian $\Delta^D_{a,K}\!\!\restriction_{|z|\leq 1}$ with Dirichlet boundary condition on $| z|= 1$ we have:
\begin{equation}\label{14:57}
\begin{aligned}
\zeta'(0,\Delta^D_{a,K}\!\!\restriction_{|z|\leq 1})=& 2\zeta_B'(0;a,1,1) -\frac {11}{12} a
\\
&+\frac 1 6 \left(a+3+\frac 1 a\right)\log a+ \frac {4}{3} \frac{a}{K+1}+\frac 1 2  \log {2\pi },
\\
\zeta(0,\Delta^D_{a,K}\!\!\restriction_{|z|\leq 1})= &\frac 1 {12}\left(a+\frac 1 a \right).
\end{aligned}
\end{equation}
 \end{lemma}
 
\begin{proof} Let us first assume that $K>1$ (we will then extend the results by analyticity with respect to $K$ to all $K>-1$ as claimed).   We intend to glue the spherical annulus 
 $
 \left( \frac 1{\sqrt[2a]{K}}\leq|z|\leq 1,  m_{a,K}\right)
 $
to the spherical  cone $\left(|z|\leq \frac 1{\sqrt[2a]{K}} , m_{a,K}\right)$ relying on the BFK decomposition formula for the determinants on surfaces with boundary~\cite[Corollary 1.3]{LeeBFKboundary},   conformal invariance of the Neumann jump operator, and Lemma~\ref{LII}. (Let us mention again that for the constant curvature metrics with conical singularities the proof of BFK formula requires only minor modifications~\cite[Section 2.3]{Kalvin Pol-Alv}).

The BFK formula for the decomposition of the determinant $\det (\Delta^D_{a,K}\!\!\restriction_{|z|\leq 1})$ along the circle $|z|= 1/{\sqrt[2a]{K}}$  reads
$$
\det (\Delta^D_{a,K}\!\!\restriction_{|z|\leq 1})=\det (\Delta^D_{a,K}\!\!\restriction_{|z|\leq \frac 1 {\sqrt[2a] K}})\cdot \det (\Delta^D_{a,K}\!\!\restriction_{\frac 1  {\sqrt[2a] K}\leq |z|\leq 1})\cdot \det ( N_{a,K}\!\!\restriction_{|z|=\frac 1 {\sqrt[2a] K}}),
$$
where $\Delta^D_{a,K}\!\!\restriction_{\frac 1  {\sqrt[2a] K}\leq |z|\leq 1}$ is the selfadjoint Dirichlet Laplacian on the annulus and $N_{a,K}\!\!\restriction_{|z|=\frac 1 {\sqrt[2a] K}}$ is the Neumann jump operator on the circle.

Let us also consider  $m_\flat=|dz|^2$ as a  reference metric. All  objects defined in terms of the flat metric $m_\flat$ will have the subscript $\flat$.  The BFK formula for the decomposition of the determinant of  Dirichlet Laplacian $\det (\Delta^D_{\flat}\!\!\restriction_{|z|\leq 1})$  along the circle $|z|= 1/{\sqrt[2a]{K}}$  reads
$$
\det (\Delta^D_{\flat}\!\!\restriction_{|z|\leq 1})=\det (\Delta^D_{\flat}\!\!\restriction_{|z|\leq \frac 1 {\sqrt[2a] K}})\cdot \det (\Delta^D_{\flat}\!\!\restriction_{ \frac 1 {\sqrt[2a] K}\leq |z|\leq 1})\cdot \det ( N_{\flat}\!\!\restriction_{|z|=\frac 1 {\sqrt[2a] K}}),
$$
where $\det ( N_{a,K}\!\!\restriction_{|z|=\frac 1 {\sqrt[2a] K}})=\det ( N_{\flat}\!\!\restriction_{|z|=\frac 1 {\sqrt[2a] K}})$ due to conformal invariance of the determinant $\det ( N_{a,K}\!\!\restriction_{|z|=\frac 1 {\sqrt[2a] K}})$ on surfaces with boundary (as noticed in~\cite{Wentworth},  the conformal invariance can be most easily seen from the BFK formula together with  Polyakov-Alvarez formula for the Dirichlet Laplacians). 
As a result, we have
\begin{equation}\label{14:34}
{\det (\Delta^D_{a,K}\!\!\restriction_{|z|\leq 1})}={\det (\Delta^D_{\flat}\!\!\restriction_{|z|\leq 1})}\cdot\frac {\det (\Delta^D_{a,K}\!\!\restriction_{|z|\leq\frac 1 { \sqrt[2a] K}})}{\det (\Delta^D_{\flat}\!\!\restriction_{|z|\leq \frac 1 {\sqrt[2a] K}})}\cdot\frac{\det (\Delta^D_{a,K}\!\!\restriction_{ \frac 1 {\sqrt[2a] K}\leq |z|\leq 1})}{ \det (\Delta^D_{\flat}\!\!\restriction_{ \frac 1 {\sqrt[2a] K}\leq |z|\leq 1})}.
\end{equation}

For the determinant of Dirichlet Laplacian on the flat  disk $(|z|\leq r, |dz|^2)$ we use the formula
\begin{equation*}
\log \det (\Delta^D_{\flat}\!\!\restriction_{|z|\leq r})=-\frac 1 3 \log r +\frac 1 3 \log 2 -2\zeta_R'(-1) -\frac 5  {12} -\frac 1 2 \log 2\pi
\end{equation*}
 found in~\cite[f-la (28)]{Weisberger}. This together with  Lemma~\ref{LII} gives
\begin{equation}\label{14:35}
\begin{aligned}
&\log\left({\det (\Delta^D_{\flat}\!\!\restriction_{|z|\leq 1})}\cdot\frac {\det (\Delta^D_{a,K}\!\!\restriction_{|z|\leq \frac 1 {\sqrt[2a] K}})}{\det (\Delta^D_{\flat}\!\!\restriction_{|z|\leq \frac 1 {\sqrt[2a] K}})}\right)=-\frac 1 {6a} \log K-2\zeta_B'(0;a,1,1)
\\
&\qquad+\frac a 4 -\frac 1 6 \left(a+3+\frac 1 a\right)\log a   + \frac 1 {12}\left(a+\frac 1 a \right)\log K  -\frac 1 2  \log {2\pi }.
\end{aligned}
\end{equation}

For the annulus $ \frac 1 {\sqrt[2a] K}\leq |z|\leq 1$ the well-known Polyakov-Alvarez formula~\cite{Alvarez,OPS} gives
\begin{equation}\label{PA}
\begin{aligned}
& \log \frac{\det (\Delta^D_{a,K}\!\!\restriction_{\frac 1  {\sqrt[2a] K}\leq |z|\leq 1})}{ \det (\Delta^D_{\flat}\!\!\restriction_{ \frac 1 {\sqrt[2a] K}\leq |z|\leq 1})}
\\
& = -\frac{1}{6\pi}\left (\frac1 2 \int_{   \frac 1 {\sqrt[2a] K}\leq |z|\leq 1  } |\nabla_\flat\psi |^2\,\frac{dz\wedge d \bar z}{-2i}+\oint_{\left\{|z|= \frac 1 {\sqrt[2a] K}\right\}\cup\{|z|=1\}}k_\flat\psi \,|d z|  \,\right )
\\
&\qquad-\frac 1 {4\pi}\oint_{\left\{|z|= \frac 1 {\sqrt[2a] K}\right\}\cup\{|z|=1\}}\partial_{n_\flat}\psi \,|d z|,
\end{aligned}
\end{equation}
where 
$$
\psi(z)=\psi(|z|)=\log (2a) +(a-1)\log |z|-\log(1+K|z|^{2a})
$$ 
is the (smooth in the annulus) potential of the metric $ m_{a,K}=e^{2\psi}|dz|^2$. As before, the symbol $\flat$ refers to the flat metric $|dz|^2$; i.e. $\nabla_\flat$ is the gradient, $k_\flat$ is the geodesic curvature
 ($k_\flat=1$ on the circle $|z|=1$ and $k_\flat= - \sqrt[2a]{K} $ on the circle $|z|= \frac 1 {\sqrt[2a]{K}}<1 $),  $\partial_{n_\flat}$ is the outer normal derivative.

For the gradient of $\psi$ with $r=|z|$ we have
 $$
 \begin{aligned}
  |\nabla_\flat\psi (z)|^2 & =|\partial_r\psi(r)|^2=\left| \frac {a-1} r -\frac {2aK r^{2a-1}}{1+Kr^{2a}}   \right|^2
  \\
 & =\frac {(a-1)^2}{r^2}-\frac {4aK(a-1)r^{2a-2}}{1+Kr^{2a}}+Kr^{2a-2}  (r\partial_r)^2\log(1+Kr^{2a}),
 \end{aligned}
 $$
 where on the last step we used the Liouville equation  for $\psi$:
  $$
\frac {4a^2 r^{2a-2}}{(1+Kr^{2a})^2}K=r^{-2}(r\partial_r)^2\log(1+Kr^{2a}).
 $$

   Evaluating the integrals in the right hand side of~\eqref{PA} we obtain
$$
\begin{aligned}
&-\frac 1 {12\pi}\int_{ \frac 1   {\sqrt[2a] K}\leq |z|\leq 1  } |\nabla_\flat\psi |^2\,\frac{dz\wedge d \bar z}{-2i}=-\frac 1 6 \int_{\frac 1 {\sqrt[2a]{K}}}^1|\partial_r\psi(r)|^2r\,dr
\\
&\ \ =-\frac 1 6 \left(\frac {(a-1)^2}{2a}\log K     -\int_{K^{-1}}^1\frac {2K(a-1)dt}{1+Kt}+ \frac {2aK^2 t^2}{1+Kt}\Bigr|_{t=K^{-1}}^{t=1}-2aK\int_{K^{-1}}^1 \frac {K t\,dt}{1+Kt}  \right)
\\
&\ \ = - \frac {(a-1)^2}{12a}\log K -\frac 1 3  \log( {1+K})  -  \frac 1 3  \frac {a}{K+1} +\frac a 6  +\frac 1 3 \log 2  ,
\end{aligned}
$$

$$
-\frac 1 {6\pi}\oint_{|z|=1}k_\flat\psi \,|d z| =-\frac 1 3 \bigl(\log( 2a)-\log(1+K)\bigr),
$$
$$
-\frac 1 {6\pi}\oint_{|z|= \frac 1 {\sqrt[2a] K}}k_\flat\psi \,|d z| =\frac 1 3 \left(\log (2a) -\frac {a-1}{2a}\log K-\log 2\right),
$$
$$
-\frac 1 {4\pi}\oint_{|z|= \frac 1 {\sqrt[2a] K}}\partial_{n_\flat}\psi \,|d z|=-\frac 1 2,\quad
-\frac 1 {4\pi}\oint_{|z|=1}\partial_{n_\flat}\psi \,|d z|=\frac{1} 2 +\frac a 2- \frac{a}{K+1}.
$$
As a result the equality~\eqref{PA} takes the form 
 $$
 \log \frac{\det (\Delta^D_{a,K}\!\!\restriction_{ \sqrt[2a] K\leq |z|\leq 1})}{ \det (\Delta^D_{\flat}\!\!\restriction_{ \sqrt[2a] K\leq |z|\leq 1})}=\frac2 3  a -\frac 4 3 \frac  {a}{K+1}-\frac 1 {12}\left(a-\frac 1 a \right)\log K. 
 $$ 
This together with~\eqref{14:34} and~\eqref{14:35} completes the proof of the first formula in~\eqref{14:57} for $K>1$. Now the equality~\eqref{14:57} extends by analyticity with respect to $K$ to all $K>-1$. 

Indeed, we only need to verify that  the  left hand side of~\eqref{14:57} is an  analytic function of $K>-1$.   In the  local polar geodesic coordinates $(r,\theta)$ centered at $z=0$ the metric~\eqref{MAK} takes the form
$
m_{a,K}=dr^2+a^2h^2(r,K)\,d\theta^2$, 
where  $\theta\in \Bbb R/2\pi \Bbb Z$ and  
\begin{equation}\label{h(r)}
h(r,K)=\left\{
\begin{array}{cc}
 K^{-1/2}\sin(K^{1/2}r), &  K>0 \ \text{(spherical cone)},\\
 r, & K=0\ \text{(flat cone)},    \\
  |K|^{-1/2}\sinh (|K|^{1/2}r), & K<0\ \text{(hyperbolic cone)}.    
\end{array}
\right.
\end{equation}
Clearly, $\lim_{r\to 0} h(r)/r=1$, $h'(0)=1$, and $h''(0)=0$. As a consequence, 
the well-known results~\cite{BS2} imply  in a standard way that  for each $K>-1$ the coefficient before $\log t$ in the  asymptotics of the heat trace $\Tr e^{-t \Delta^D_{a,K}\!\!\restriction_{|z|\leq 1}}$ as $t\to 0+$ is zero and the spectral zeta function $\zeta(s,\Delta^D_{a,K}\!\!\restriction_{|z|\leq 1})$ continues by analyticity with respect to $s$ from $\Re s>2$ to a neighborhood of $s=0$;  details can be found e.g. in~\cite[Sec. 5]{KalvinJGA}. Since $-1<K\mapsto \Delta^D_{a,K}\!\!\restriction_{|z|\leq 1}$ is an analytic family in the sense of Kato~\cite{Kato} and $(\Delta^D_{a,K}\!\!\restriction_{|z|\leq 1})^{-2}$ is in the trace class, the representation
$$
\zeta(s,\Delta^D_{a,K}\!\!\restriction_{|z|\leq 1})=\frac 1 {2\pi i (s-1)}\int_{\mathcal C} \lambda^{1-s}\Tr(  \Delta^D_{a,K}\!\!\restriction_{|z|\leq 1}-\lambda )^{-2}\,d\lambda
$$
(here $\mathcal C$ is a contour running clockwise at a sufficiently close distance around the cut $(-\infty,0]$ and  $\lambda^z=|\lambda|^z e^{iz\arg z}$ with $|\arg \lambda|\leq \pi$~\cite{Shubin}) implies that $\zeta(s,\Delta^D_{a,K}\!\!\restriction_{|z|\leq 1})$ is an analytic function of the variables $s$ and $K$ for $\Re s>2$ and $K>-1$. Therefore the function $(s,K)\mapsto \zeta(s, \Delta^D_{a,K}\!\!\restriction_{|z|\leq 1})$ continues by analyticity from $\Re s>2$, $K>-1$ to a neighborhood of $s=0$ by a classical result due to Friedrich Hartogs. In particular,  $\zeta'(0, \Delta^D_{a,K}\!\!\restriction_{|z|\leq 1})$ is the real analytic function of $K>-1$ found in~\eqref{14:57}.

In order to verify the second equality in~\eqref{14:57} one can, for instance, repeat all calculations  for the scaled metric $C\cdot  m_{a,K}$,  find $\log\det(C^{-1} \Delta^D_{a,K}\!\!\restriction_{|z|\leq 1})$, and then extract the expression for $\zeta(0, \Delta^D_{a,K}\!\!\restriction_{|z|\leq 1})$ from the standard rescaling property
$$
\log\det(C^{-1} \Delta^D_{a,K}\!\!\restriction_{|z|\leq 1})=\log\det( \Delta^D_{a,K}\!\!\restriction_{|z|\leq 1}) -\zeta(0, \Delta^D_{a,K}\!\!\restriction_{|z|\leq 1})\log C.
$$
This is straightforward and we omit the details.

Let us note that in the particular case of $K=0$ the formulas~\eqref{14:57} are in agreement with the results of~\cite{Spreafico}, they are also in agreement with the results obtained in~\cite[Sec. 3.3]{Kalvin Pol-Alv} by a completely different method. 
\end{proof}

Now we are in position to prove the main result of this paper. 

\begin{proof}[Proof of Theorem~\ref{main}] The change of variable $z\mapsto|K|^{-\frac 1 {2a}} z$ with $-1<K<0$ shows that the hyperbolic cones $(|z|\leq 1,|K|\cdot m_{a,K})$ and $(|z|\leq |K|^{\frac 1 {2a}}, m_{a,-1})$
are isometric. By the usual rescaling property we thus obtain
\begin{equation}\label{rescl}
\log\det \Delta^D_{a,-1}\!\!\restriction_{|z|\leq |K|^{\frac 1 {2a}}}=\log\det \Delta^D_{a,K}\!\!\restriction_{|z|\leq 1}-\zeta(0,\Delta^D_{a,K}\!\!\restriction_{|z|\leq 1})\log|K|,
\end{equation}
where $ \Delta^D_{a,-1}\!\!\restriction_{|z|\leq |K|^{\frac 1 {2a}}}$ stands for the Friederichs selfadjoint extension of the Dirichlet Laplacian on the hyperbolic (curvature $-1$) cone $(|z|\leq |K|^{\frac 1 {2a}}, m_{a,-1})$. 

In the geodesic polar coordinates $(r,\theta)$ centered at $z=0$ both hyperbolic cones mentioned above take the form 
$\left( [0,\eta]_r\times [0,2\pi)_\theta, m_a\right)$ with the metric $m_a$ from~\eqref{ma} and $\eta>0$ satisfying 
\begin{equation}\label{Keta}
K=
-\tanh^2\frac \eta 2.
\end{equation}  
For instance,  the geodesic polar coordinates $(r,\theta)$ are related to the coordinate $z$  of  the hyperbolic cone $(|z|\leq |K|^{\frac 1 {2a}}, m_{a,-1})$  by the equalities
$$
r=2\tanh^{-1}|z|^a,\quad \theta=\arg z;
$$
in particular, $\eta=2\tanh^{-1}\sqrt{|K|}$, cf.~\eqref{Keta}. (Note that in the case $a=1$ there is no conical singularity,  see Remark~\ref{pol-alv} below.)

Since for $K$ and $\eta$ related by~\eqref{Keta} the hyperbolic cones $\left( [0,\eta]_r\times [0,2\pi)_\theta, m_a\right)$ and $(|z|\leq |K|^{\frac 1 {2a}}, m_{a,-1})$ are isometric, the equality~\eqref{rescl} gives
$$
\log\det\Delta^D_{a}\!\!\restriction_{r\leq \eta}=\log\det \Delta^D_{a,-1}\!\!\restriction_{|z|\leq |K|^{\frac 1 {2a}}}=\log\det \Delta^D_{a,K}\!\!\restriction_{|z|\leq 1}-\zeta(0,\Delta^D_{a,K}\!\!\restriction_{|z|\leq 1})\log|K|.
$$
This together with Lemma~\ref{LIII} and  relation~\eqref{Keta} completes the proof of Theorem~\ref{main}.
\end{proof}

Next we prove Corollary~\ref{cor} that deals with the orbifold setting: a hyperbolic cone of angle $2\pi/w$ with a positive integer $w$. 
\begin{proof}[Proof of Corollary~\ref{cor}] Let $w$ be a positive integer. We only need to show that 
\begin{equation}\label{zeta}
\zeta_B'\left(0; 1 /w,1,1\right)=\frac 1 w \zeta'_R(-1)-\frac 1 {12w}\log w-\frac 1 w \sum_{j=1}^{w-1} j \log\Gamma\left(\frac j w\right)+\frac{w-1}4 \log 2\pi,
\end{equation}
where $\zeta'_R$ stands for the derivative of the Riemann zeta function. Then~\eqref{zeta} together with~\eqref{star} implies~\eqref{w} and thus completes thee proof of Corollary~\ref{cor}. 

The identity~\eqref{zeta} is a particular case of a more general result proved in~\cite[Lemma A.1]{Kalvin Pol-Alv}.  For the reader's convenience and to make this paper self-contained we give a proof of~\eqref{zeta} below.

We rely on the definition~\eqref{BZ} of the Barnes double zeta function and notice that
$$\sum_{k=1}^w\zeta_B(s;w,1,k)=\zeta_B(s;1,1,1)=\sum_{\ell=0}^\infty\sum_{n=0}^\ell(\ell+1)^{-s}=\zeta_R(s-1).$$
 For each term of the sum on the left we have
\begin{equation}\label{auxil1}
\begin{aligned}
&\zeta_B(s;w,1,k)
=\sum_{m=0}^\infty\sum_{n=k-1}^\infty(wm+n+1)^{-s}
\\
&=\sum_{m=0}^\infty\left(\sum_{n=0}^\infty(wm+n+1)^{-s}-\sum_{n=0}^{k-2}(wm+n+1)^{-s}\right)
\\
&=\zeta_B(s;w,1,1)-w^{-s}\sum_{j=1}^{k-1}\zeta_H(s;j/w),
\end{aligned}
\end{equation}
where $\zeta_H(s; x )=\sum_{m=0}^\infty(m+x)^{-s}$ is the Hurwitz zeta function. Thus we obtain
\begin{equation}\label{auxil2}
\begin{aligned}
\zeta_B(s;w,1,1)&=\frac 1 w \zeta_R(s-1)+w^{-s-1}\sum_{k=1}^w\sum_{j=1}^{k-1}\zeta_H(s;j/w)
\\&=\frac 1 w \zeta_R(s-1)+w^{-s-1}\sum_{j=1}^{w-1}(w-j)\zeta_H(s;j/w).
\end{aligned}
\end{equation}
Since 
$\zeta_B(s;1/w,1,1)=w^{s}\zeta_B(s;w,1,w)$, thanks to~\eqref{auxil1} with $k=w$ we have
$$
\zeta_B(s;1/w,1,1)=w^s\left(\zeta_B(s;w,1,1)-w^{-s}\sum_{j=1}^{w-1}\zeta_H(s;j/w)\right).
$$
This together with~\eqref{auxil2} gives
$$
\zeta_B(s;1/w,1,1)=w^{s-1}\zeta_R(s-1)-\frac 1 w\sum_{j=1}^{w-1}j \zeta_H\left(s; j/ w\right). 
$$

Now we differentiate with respect to $s$. Taking into account that 
$$
\zeta_R(-1)=-\frac 1{12},\quad \zeta_H'(0;j/w)=\log\Gamma\left (\frac j w\right)-\frac 1 2 \log 2\pi,
$$
we arrive at~\eqref{zeta}. 
\end{proof}

\begin{remark}\label{pol-alv} In the particular case of $a=1$ the hyperbolic cone $\left( [0,\eta]_r\times [0,2\pi)_\theta, m_a\right)$, where $m_a$ is the metric~\eqref{ma},   has no conical singularity (the cone angle is $2\pi$). It is isometric to the disk 
$$\left(|z|\leq {\tanh\frac\eta 2 }<1,e^{2\psi(z)}|dz|^2\right),\quad \psi(z)=\log 2-\log(1-|z|^2),$$
 cut out of the usual Poincar\'e disk $(|z|\leq 1, e^{2\psi(z)}|dz|^2)$. The isometry is given by 
 $$
r=2\tanh^{-1}|z|,\quad \theta=\arg z.
 $$
Therefore the determinant $\det\Delta^D_1\!\!\restriction_{r\leq \eta}$ of the selfadjoint Dirichlet Laplacian  on the hyperbolic cone $\left( [0,\eta]_r\times [0,2\pi)_\theta, m_1\right)$ can be easily found by using the usual Polyakov-Alvarez formula~\cite{Alvarez,OPS} together with explicit formula for the determinant of  Dirichlet Laplacian for a reference metric. 
 
 Indeed, let us  take the flat metric $m_\flat=|dz|^2$ as the reference metric on the disk $|z|\leq {\tanh\frac\eta 2 }$. Then  the Polyakov-Alvarez formula reads
 \begin{equation}\label{PAL}
 \begin{aligned}
 \log \frac {\det(\Delta^D_1\!\!\restriction_{r\leq \eta})} {\det (\Delta^D_\flat\!\!\restriction_{|z|\leq {\tanh\frac\eta 2 }})}=& -\frac{1}{6\pi}\left (\frac1 2 \int_{|z|\leq {\tanh\frac\eta 2 }} |\nabla_\flat\psi|^2\frac{dz\wedge d\bar z}{-2i}\right.
 \\
&\left. +\oint_{|z|={\tanh\frac\eta 2 }}k_\flat\psi \,|d z|  \,\right )
 -\frac 1 {4\pi}\oint_{|z|={\tanh\frac\eta 2 }}\partial_{n_\flat}\psi\,|dz|,
 \end{aligned}
\end{equation}
where $\nabla_\flat$ is the gradient, $k_\flat=1/\tanh\frac\eta 2 $ is the geodesic curvature,  and $\partial_{n_\flat}=\partial_{|z|}$ is the outer normal derivative, all with respect to the  metric $m_\flat$. By~\cite[f-la (28)]{Weisberger} we have
$$
\log \det (\Delta^D_{\flat}\!\!\restriction_{|z|\leq  {\tanh\frac\eta 2 }})=-\frac 1 3 \log  {\tanh\frac\eta 2 }  +\frac 1 3 \log 2 -2\zeta_R'(-1) -\frac 5  {12} -\frac 1 2 \log 2\pi.
$$
Evaluation of the integrals in~\eqref{PAL} gives
$$
\begin{aligned}
 \int_{|z|\leq {\tanh\frac\eta 2 }} |\nabla_\flat\psi|^2\frac{dz\wedge d\bar z}{-2i}&=2\pi\int_0^{\tanh \frac \eta 2}\frac {4|z|^3\,d|z|}{(1-|z|^2)^2}
 \\
 &=4\pi\left( \log\left(1-\tanh^2\frac\eta 2 \right)+\left(1-\tanh^2\frac \eta 2\right)^{-1}-1\right),
\end{aligned}
$$
$$
\oint_{|z|={\tanh\frac\eta 2 }}k_\flat\psi \,|d z|=2\pi\left(\log 2-\log\left(1-\tanh^2\frac\eta 2 \right)\right),
$$
$$
\oint_{|z|={\tanh\frac\eta 2 }}\partial_{n_\flat}\psi\,|dz|=\int_{|z|={\tanh\frac\eta 2 }}\frac{2|z|}{1-|z|^2}\,|dz|=4\pi\left(\left(1-\tanh^2\frac \eta 2\right)^{-1}-1\right).
$$
As a result, from~\eqref{PAL} we obtain
$$
\log {\det(\Delta^D_1\!\!\restriction_{r\leq \eta})} =-\frac 1 3 \log  {\tanh\frac\eta 2 }  -2\zeta_R'(-1) +\frac {11}  {12} -\frac 4 3  \left(1-\tanh^2\frac \eta 2\right)^{-1} -\frac 1 2 \log 2\pi,
$$
which is in agreement with Theorem~\ref{main}, Corollary~\ref{cor}, and the asymptotics~\eqref{correction}, but not with the asymptotics
$$
\begin{aligned}
&\log  {\det(\Delta^D_{w^{-1}}\!\!\restriction_{r\leq \eta})}=-\left(\frac w 6 +\frac 1 {6w}\right)\log(\eta)-w\left(-2\zeta_R'(-1)+\frac 1 6 -\frac 1 6 \log 2 \right)
\\
&-\frac 1 w \left(\frac 5 {12} -\frac 1 6\log 2 +\frac \gamma 6\right)+\frac 1 2\log w+\frac 1 6 w\log w+\frac 1 6 \frac {\log w}{w}+\frac 1 4+o(1),\quad \eta\to0^+,
\end{aligned}
 $$
from~\cite[Theorem~2.5]{F-P}, cf.~\eqref{correction}. This  is probably the most simple and convincing way of showing that the results in~\cite[Theorem~2.5 and Section 7]{F-P} are incorrect, which also affects the main results of that paper. 
\end{remark}


\begin{thebibliography}{100}
\bibitem{Alvarez} O. Alvarez, Theory of strings with boundary, Nucl. Phys. B 216 (1983), 125--184
  \bibitem{F-P} G. Freixas i Montplet, A. von Pippich, Riemann-Roch isometries in the non-compact orbifold setting, J. Eur. Math. Soc. 22 (2020), 3491-3564,  
   \href{https://arxiv.org/abs/1604.00284}{\tt arXiv:1604.00284 [math.NT]}
  
  
\bibitem{BS2} { J. Br\"uning   and  R. Seeley}, The resolvent expansion for second order regular singular operators, J. Funct. Anal.  73 (1987),  369--429 

\bibitem{BFK} D. Burghelea, L. Friedlander, and T. Kappeler, Meyer-Vietoris type formula for determinants of elliptic differential operators, J. Funct. Anal. 107 (1992), 34--65 

 \bibitem{HKK1} L. Hillairet, V. Kalvin, A. Kokotov, Spectral determinants on Mandelstam diagrams,  Comm. Math. Phys. 343 (2016), no. 2, 563--600,   \href{https://arxiv.org/abs/1312.0167}{\tt arXiv:1312.0167 [math.SP]}
  
 \bibitem{HKK2} L. Hillairet, V. Kalvin, A. Kokotov, Moduli spaces of meromorphic functions and determinant of Laplacian,  Trans. Amer. Math. Soc. 370 (2018), no. 7, 4559--4599, 
 \href{https://arxiv.org/abs/1410.3106}{\tt arXiv:1410.3106 [math.SP]}

\bibitem{Kalvin Pol-Alv} V. Kalvin, Polyakov-Alvarez type comparison formulas for determinants of Laplacians on Riemann surfaces with conical singularities, Preprint (2019), 
 \href{https://arxiv.org/abs/1910.00104}{\tt arXiv:1910.00104 [math-ph]}




\bibitem{KalvinJGA} V. Kalvin, On determinants of Laplacians on compact Riemann surfaces equipped with pullbacks of conical metrics by meromorphic functions.  J. Geom. Anal. 29 (2019), pp. 785--798, 
 \href{https://arxiv.org/abs/1712.05405}{\tt arXiv:1712.05405 [math.AP]}



 


 \bibitem{Kato} T. Kato, Perturbation theory for linear operators. Springer, Berlin (1966)


\bibitem{Klevtsov} S. Klevtsov, Lowest Landau level on a cone and zeta determinants, J.Phys. A: Math. Theor. 50 (2017), 	 \href{https://arxiv.org/abs/1609.08587}{\tt arXiv:1609.08587 [cond-mat.str-el]}

\bibitem{Kokot} A. Kokotov, Polyhedral surfaces and determinant of Laplacian. Proc. Amer. Math. Soc., 141 (2013), 725--735,  \href{https://arxiv.org/abs/0906.0717}{\tt arXiv:0906.0717 [math.DG]}
  
\bibitem{LeeBFKboundary} Y. Lee, Burghelea-Friedlander-Kappeler's gluing formula for the zeta-determinant and its applications to the adiabatic decompositions of the zeta-determinant and the analytic torsion. 
Trans. Amer. Math. Soc. 355 (2003), no. 10, 4093--4110 

\bibitem{OPS} B. Osgood, R. Phillips, P. Sarnak, Extremals of Determinants of Laplacians. J.~Funct.~Anal. 80 (1988), 148--211


\bibitem{Shubin} {M. Shubin}, Pseudodifferential operators and spectral theory, Springer, 2001

 
\bibitem{Spreafico} M. Spreafico, Zeta function and regularized determinant on a  disk and on a cone,  Journal of Geometry and Physics 54 (2005), 355--371

\bibitem{Spreafico3}  M. Spreafico, S. Zerbini, Spectral analysis and zeta determinant on the deformed spheres,  Commun. Math. Phys. 273 (2007), 677--704

\bibitem{Spreafico2} M. Spreafico, On the Barnes double zeta and Gamma functions,  J. Number Theory 129 (2009), 2035--63

\bibitem{TZ} L.A. Takhtajan, P.G.  Zograf,  Local index theorem for orbifold Riemann surfaces. Lett. Math. Phys. 109 (2019), pp. 1119--1143, \href{https://arxiv.org/abs/1701.00771}{arXiv:1701.00771 [math.AG]}

\bibitem{Teo} L.P. Teo, Ruelle zeta function for cofinite hyperbolic Riemann surfaces with ramification points. Lett. Math. Phys. 110 (2020), pp. 61--82, \href{https://arxiv.org/abs/1901.07898}{arXiv:1901.07898 [math.NT]}

\bibitem{Troyanov} M. Troyanov, Metrics of constant curvature on a sphere with two conical singularities, Lecture Notes in Math. 1410 (1989), 296--306

 \bibitem{Wentworth}  R. Wentworth, Precise constants in bosonization formulas on Riemann surfaces,  Commun. Math. Phys. 282 (2008), 339--355


\bibitem{Weisberger}  W. Weisberger, Conformal invariants for determinants of Laplacians on Riemann surfaces, Commun. Math. Phys. 112 (1987), 633--638

\end{thebibliography}
\end{document}